\documentclass[10pt]{amsart}
\usepackage{graphicx}
\usepackage{indentfirst,csquotes}

\usepackage[centertags, reqno, sumlimits, nointlimits, namelimits]{mathtools}
\usepackage{amsthm} 
\usepackage{amssymb} 
\usepackage{amsfonts} 
\usepackage{mathrsfs}
\usepackage{nicefrac} 

\usepackage{xcolor,paralist,fancyhdr,etoolbox}

\usepackage{hyperref}
\usepackage[english]{varioref}
\usepackage[english, noabbrev]{cleveref}

\hypersetup{ colorlinks=true, linkcolor=black, filecolor=black, urlcolor=black }

\theoremstyle{definition}
\newtheorem{defi}{Definition}

\theoremstyle{plain}

\newtheorem{lemm}[defi]{Lemma}
\newtheorem{coro}[defi]{Corollary}
\newtheorem{theo}[defi]{Theorem}
\newtheorem*{theo*}{Theorem}
\newtheorem{prop}[defi]{Proposition}

\theoremstyle{remark}
\newtheorem{rema}[defi]{Remark}
\newtheorem*{rema*}{Remark}

\newcommand{\N}{\mathbb{N}}
\newcommand{\R}{\mathbb{R}}

\newcommand{\Z}{\mathbb{Z}}
\newcommand{\C}{\mathbb{C}}

\newcommand{\ph}{\varphi}

\renewcommand{\Re}{\mathrm{Re}}
\renewcommand{\Im}{\mathrm{Im}}

\newcommand{\Ha}{\mathbb{H}}
\DeclareMathOperator{\form}{\mathfrak{t}}
\DeclareMathOperator{\grad}{\mathrm{grad}}

\begin{document}
	\title{Perturbation Theory for Pseudo-Laplacians}
	\author[Initial Surname]{Marlis Balkenhol}
	\date{\today}
	\address{Mathematisches Institut der Heinrich-Heine-Universit\"at D\"usseldorf, Universit\"atsstr. 1, 40225 Germany}
	\email{Marlis.Balkenhol@hhu.de}
	\maketitle
	
	\let\thefootnote\relax
	\footnotetext{This is a preprint. It has not yet been peer-reviewed or accepted for publication, and does not contain any modifications that may be requested after submission.}
	
	\begin{abstract}
		We generalise the notion of the Pseudo-Laplacian on a hyperbolic Riemann surface with one cusp, that was studied by Lax and Phillips and Colin de~Verdi\`ere, 
		by considering a boundary condition of Robin type for the constant term instead of the classical Dirichlet condition.
		The resulting family of Pseudo-Laplacians is a holomorphic family of unbounded operators in the sense of Kato. 
		By use of holomorphic perturbation theory, 
		we study the eigenvalues and eigenvectors of this family and give a new proof for the meromorphic continuation of the Eisenstein series.  
	\end{abstract} 
	
	\bigskip

	\section{Introduction}
	
	Let~$X$ be a hyperbolic Riemann surface with one cusp at~$\infty$. $X$ can be represented by the quotient~$\Gamma\backslash\Ha$ of the hyperbolic upper half plane~$\Ha$ by the action of a Fuchsian group~$\Gamma$.
	The functions on~$X$ are called \textit{automorphic} as they are invariant under the group action of~$\Gamma$. 
	Denote by 
	\begin{equation*}
		\Delta = -y^2 (\partial_{xx} + \partial_{yy})
	\end{equation*}
	the Laplace operator on~$X$ in the sense of distributions and let~$\Delta_\Gamma$ be the self-adjoint operator constructed by the Friedrichs extension of~$\Delta$ restricted to~$\mathcal{D}(X)$. 
	
	The \textit{Eisenstein series}, defined by
	\begin{equation}\label{eisen}
		E(z,s) := \sum_{\gamma \in \Gamma_\infty\backslash\Gamma}{(\Im{(\gamma.z)})^s}
	\end{equation}
	where~$\Gamma_\infty$ is the stabiliser of the cusp~$\infty$, is an automorphic function on~$X$.
	The series in~\eqref{eisen} converges absolutely if and only if~$\Re{s}>1$, but~$E(z,s)$ can be continued meromorphically for all~$s \in \C$.
	
	In the terminology of Kubota \cite{Kubota1973}, the Eisenstein series is a 'generalised eigenfunction' of~$\Delta$ as~$\Delta E(z,s) = s(1-s) E(z,s)$  for all~$z \in \Ha$.
	$E(\, \cdot \, ,s)$ is no proper eigenvector of~$\Delta_\Gamma$, but it is connected to the continuous spectrum of~$\Delta_\Gamma$ and the (discrete) spectrum of the \textit{Pseudo-Laplace operator}~$\Delta^\eta$ studied in~\cite{cdv1},~\cite{cdv2} and~\cite{Lax1976}. 
	As a specific result, Colin de~Verdi\`ere \cite{cdv2} re-proves Selberg's theorem of the meromorphic continuation of the Eisenstein series by meromorphy of the resolvent of~$\Delta^\eta$.
	
	In this paper, we modify the construction of the Pseudo-Laplacian~$\Delta^\eta$, gaining for each~$\eta$ a new family~$(\Delta^\eta_\gamma)_{\gamma\in\hat{\C}}$ of operators with compact resolvent, and study its spectrum with holomorphic perturbation theory.
	Based on the holomorphy properties of the family~$(\Delta^\eta_\gamma)_{\gamma\in\hat{\C}}$, we give yet another proof of the meromorphic continuation theorem for the Eisenstein series. \\
	
	\noindent 
	\textbf{Acknowledgement.}
	This paper is based on the author's PhD thesis \cite{balk2022}. The author would like to thank her supervisor R\"udiger Braun for many helpful discussions.
	
	\section{Classical Pseudo Laplacians}\label{class-plo}
	
	We start with a short overview of the construction and results on the spectrum of~$\Delta^\eta$ in~\cite{cdv2}. 
	
	Let the surface~$X$ be split into two disjoint submanifolds~$X_\eta$ and~$S_\eta$ with boundary where~$X_\eta$ is compact and~$S_\eta$ is a cusp neighbourhood. 
	In the half-plane model, we may think of~$X_\eta$ containing all those~$z\in F$ with~$\Im(z) \leq \eta$ and~$S_\eta$ all those~$z\in F$ with~$\Im{z}>\eta$ after choosing a fundamental polygon~$F$ of~$X$ with a cusp at~$\infty$. In this setting, the common boundary of~$X_\eta$ and~$S_\eta$ is the line segment~$\mathcal{H}_\eta \simeq \, ]0,1[ \times \{\eta\}$ of the horocycle through~$\infty$ and~$\eta$.
	
	Due to the geometry of~$F$, there is a lower bound~$p$, depending on the group~$\Gamma$, for all admissible choices of~$\eta$ such that~$X$ can be reasonably split in the way just described.
	For our construction of~$\Delta_\gamma^\eta$, we will always assume that~$\eta > p$.

	Now let~$\mathscr{H}_\eta = \{ f \in H^1(X) \ | \ f_0|_{]\eta,\infty[}=0 \}$ where~$f_0$ is the constant term of the Fourier series~$\sum_{m\in\Z}{f_m(y) e(mx)}$, $e(mx)=e^{2\pi i mx}$, of~$f$.
	Note that the condition~$f \in H^1(X)$ in this definition implies that~$f_0 \in C(]p,\infty[)$ with~$f_0(\eta)=0$.
	
	The Pseudo-Laplacian~$\Delta^\eta$ is defined as the Friedrichs extension of the distributional Laplacian~$\Delta$ restricted to~$\mathscr{H}_\eta$. $\Delta^\eta$ is a positive, self-adjoint operator in~$L^2_\eta = \{ f \in L^2(X) \ | \ f_0|_{]\eta,\infty[}=0 \}$
	with compact resolvent 
	(\cite{cdv2}, Theorem~2). \\
	
	\noindent 
	Let~$L^2_0$ be the subspace of~$L^2(X)$ consisting of all~$f$ with trivial constant term~$f_0=0$, and~$\Theta$ the orthogonal complement of~$L^2_0$ in~$L^2(X)$ both of which are~$\Delta$-invariant.
	The classical Laplacian~$\Delta_\Gamma$ has a finite or infinite discrete spectrum in~$L^2_0$, the corresponding eigenfunctions are called cusp forms.
	The spectrum of~$\Delta_\Gamma$ in~$\Theta$ consists of a continuous part in~$[\frac{1}{4},\infty[$ and finitely many eigenvalues in~$[0,\frac{1}{4}[$, arising from the poles of~$E(\, \cdot \, ,s)$ in~$]\frac{1}{2},1]$ (\cite{Kubota1973}, Theorem~5.2.4). 
	
	As~$L^2_0 \subset L^2_\eta$, we define~$\Theta_\eta$ as the orthogonal complement of~$L^2_0$ in~$L^2_\eta$, which is also~$\Delta$-invariant. 
	According to this decomposition, the spectrum of~$\Delta^\eta$ may be split into two parts~$\sigma_1: = \sigma(\Delta^\eta|_{L^2_0})$ and~$\sigma_2(\Delta^\eta):=\sigma(\Delta^\eta|_{\Theta_\eta})$ which are not necessarily disjoint.
	We call the eigenvalues associated with a cusp form of \textit{type~(I)} and all other eigenvalues of \textit{type~(II)}. 
	
	The type-(II) eigenvalues of~$\Delta^\eta$ form an unbounded sequence~$0 < \mu_1 < \mu_2 <  \ldots$ where the~$\mu_j=s_j (1-s_j)$ are determined by those~$s_j$ for which the Eisenstein series fulfils the smoothness condition
	\begin{equation}\label{smooth-at-eta}
		E_0(\eta,s) = 0
	\end{equation} 
	where~$E_0(y,s) =  y^s + \ph(s) y^{1-s}$.
	If~$s_j\neq\frac{1}{2}$, the eigenfunction is given by the \textit{truncated Eisenstein series}
	\begin{equation*}
		E^\eta(z,s_j):= \begin{cases}
			E(z,s_j), & \Im{z} \leq \eta, \\
			E(z,s_j) - E_0(y,s_j), & \Im{z} >\eta,
		\end{cases}
	\end{equation*}
	that is gained by 'cutting off' the constant term at~$\eta$ and is an element of~$L^2_\eta$.
	In the special case~$s_j=\frac{1}{2}$, it may happen~$E(\, \cdot \, ,s)$ vanishes everywhere on~$\Ha$. In this case, an eigenfunction is given by~$\partial_s E^\eta(\, \cdot \, , s)|_{s=\nicefrac{1}{2}}$ if and only if~$\partial_s E_0(\eta,s)|_{s=\nicefrac{1}{2}} =0$ (\cite{cdv2}, Theorem~5).

	\section{Robin Pseudo Laplacians}
	
	We interpret the smoothness condition \eqref{smooth-at-eta} of~$\Delta^\eta$ as a Dirichlet boundary condition for the constant term~$v_0$ of any~$v \in L^2_\eta$ on the horocycle $\mathcal{H}_\eta$. Now, while Colin de~Verdi\`ere considers the dependence of~$\Delta^\eta$ on~$\eta$, we fix one~$\eta>p$, and consider a more general boundary condition of Robin type
	\begin{equation}\label{robin}
		v_0'(\eta) = -\gamma v_0(\eta), \qquad \gamma\in\C.
	\end{equation}
	
	Let $A_1 = \{ f \in L^2_\eta \, | \, \exists \, C_f \in \C: \grad f - C_f \cdot \delta_{\eta,0} \in L^2_\eta \}$ where~$\grad$ is the distributional gradient on~$X$ and~$\delta_{\eta,0}$ is the distribution defined by~$\left<\delta_{\eta,0},\ph\right> = \ph_0(\eta)$. On~$A_1$, we define a sesquilinear form by
	\begin{equation*}
		\form_\gamma(u,v) = (u,v)_{A} + \gamma \cdot u_0(\eta) \bar{v}_0(\eta), \qquad u,v \in A_1,
	\end{equation*}
	where~$(u,v)_A = (u,v)_{L^2_\eta} + (\grad u - C_{u} \delta_{\eta,0}, \grad v - C_v \delta_{\eta,0})_{L^2_\eta}$.
	Using Ehrling's lemma, one sees that the form~$\form_\gamma$ is sectorial for every~$\gamma \in \C$. 
	The Representation Theorem~VII-2.1 in \cite{Kato1976} provides that there is a unique m-sectorial operator~$T_\gamma$ with domain~$D(T_\gamma)$ dense in~$A_1$ such that~$\form_\gamma(u,v) = (T_\gamma u,v)$ for all~$u\in D(T_\gamma)$, $v \in A_1$. 
	
	\begin{defi}
		The \textit{Robin Pseudo Laplacian}~$\Delta_\gamma$ is defined by
		\begin{equation*}
			\Delta_\gamma := T_\gamma - 1.
		\end{equation*}
	\end{defi}
	As~$\eta$ is usually fixed, the dependence of~$\Delta_\gamma$ on~$\eta$ is expressed by writing~$\Delta^\eta_\gamma$ only when it is important.
	
	\begin{theo} $\Delta_\gamma$ has the following properties:
		\begin{enumerate}[(i)]
			\item For every~$\gamma\in\C$, $\Delta_\gamma$ has compact resolvent.
			\item $\Delta_\gamma$ is self-adjoint if and only if~$\gamma\in\R$. 
		\end{enumerate}
	\end{theo}
	
	\begin{proof}
		The proof of~(i) is analogous to the proof for~$\Delta^\eta$ which goes back to \cite{Lax1976}, Lemma~8.7. The second property follows from the symmetry of~$\form_\gamma$ with respect to~$\gamma$.
	\end{proof}
	
	\section{Analyticity}
	
	Comparing the Robin boundary condition with the original smoothness condition, it is easily seen that, from the new point of view,~$\Delta^\eta$ corresponds to the point at infinity if we allow~$\gamma\in\hat{\C}$.
	We write~$\Delta^\eta_\infty$ (or shortly~$\Delta_\infty$) for the classical Pseudo-Laplacian~$\Delta^\eta$, hence having a family~$(\Delta^\eta_\gamma)_{\gamma\in\hat{\C}}$ of operators for each~$\eta>p$.
	
	One can show that the family~$(\Delta_\gamma)_{\gamma\in\hat{\C}}$ is a holomorphic family of (unbounded) operators. We refer to \cite{Kato1976}, VII-\S 1.2, for precise definition, as we are mostly interested in the analyticity properties of the (discrete) spectrum of~$(\Delta_\gamma)_\gamma$.
	
	If~$T_z$ is an operator between finite dimensional spaces depending on some complex parameter~$z$, the spectrum of~$T_z$ is identical to the zero set of its characteristic polynomial. 
	Hence the set~$\mathcal{C}=\{ (z,\lambda) \ | \ \lambda  \in \sigma(T_z), z \in U\}$ is a complex variety if the family~$(T_z)_{z\in U}$ is holomorphic on some open set~$U\subset\C$. 
	By separating the spectrum in a specific way, Kato shows that this statement still holds for a holomorphic family of unbounded operators~$(T_z)_z$ if the whole spectrum~$\sigma(T_z)$ in the definition of~$\mathcal{C}$ is replaced by any discrete subset of~$\sigma(T_z)$ (\cite{Kato1976}, VII-\S 1.3).
	
	In particular, all eigenvalues and eigenprojections of~$T_z$ are locally given by branches of analytic functions which means that they are holomorphic except for at most algebraic branch points. Such branch points, however, can only occur in~$T_z$ is not self-adjoint.

	The analyticity of the eigenprojections allows to find local analytic (vector-valued) functions~$z\mapsto v[z]$ of the eigenfunctions of~$T_z$. The construction of such a function is explained for the finite dimensional case and a complete orthonormal family in \cite{Kato1976}, II-\S 6.2, and can be generalised for any single eigenvector of an unbounded operator with the results in \cite{Kato1976}, VII-\S 1.3. 
	Note, however, that it is in general not possible to find such local analytic representation functions uniformly for the set of all eigenvalues and -vectors.

	\section{Uniqueness of Eigenvalues}
	
	As for~$\Delta^\eta$, we call an eigenvalue or eigenfunction~$f$ of~$\Delta_\gamma$ of \textit{type~(I)} if~$f \in L^2_0$ and of \textit{type~(II)} if~$f \in \Theta_\eta$.
	The sequence of type-(I) eigenvalues does not change with~$\gamma$ nor~$\eta$; the type-(II) eigenvalues will always be denoted by 
	\begin{equation}
		\lambda = s \cdot \hat{s}, \qquad  \hat{s}=1-s,  \qquad s \in \C,
	\end{equation}
	the set of type-(II) eigenvalues of~$\Delta_\gamma$ by~$\sigma_2(\Delta_\gamma)$. 
	We write~$\lambda(s)$ when we consider~$\lambda$ as a holomorphic function of~$s$. 
	
	The constant term of a corresponding eigenfunction~$v$ has the form
	\begin{equation}\label{const-term}
		v_0(y) = 
		\begin{cases}
			a y^s + b y^{1-s}, & s \neq \frac{1}{2}, \\
			a \sqrt{y} + b \ln(y) \sqrt{y}, & s = \frac{1}{2},
		\end{cases} 
	\end{equation}
	for complex numbers~$a$ and $b$ not vanishing simultaneously. All other Fourier coefficients~$v_m$, $m\in \Z\backslash\{0\}$, of~$v$ are given by
	$v_m(y) = a_m \sqrt{y} K_{s-\nicefrac{1}{2}}(2\pi|m|y)$ with $a_m \in \C$ and the modified Bessel function~$K_\nu$ of the second kind. 
	
	The invariance of~$\lambda$ under the involution~$z\mapsto \hat{z}$ assigns a certain symmetry of the coefficients~$a=a[s]$ and~$b=b[s]$ in~\eqref{const-term} with respect to~$s$. 
	Hence, every eigenfunction~$v[\lambda]$ of an eigenvalue~$\lambda$ gives two functions~$v[s]$ and~$v[\hat{s}]$ where~$a[s]=b[\hat{s}]$ and~$b[s]=a[\hat{s}]$. 
	The symmetry of~$K_\nu$ with respect to~$\nu$ implies that~$a_m[s]=a_m[\hat{s}]$ for~$m\neq 0$.\\ 
	
	\noindent  
	An argument of energy conservation (see~\cite{Lax1976}, Theorem~8.4, for~$s\neq \frac{1}{2}$, \cite{cdv2}, Theorem~5, for~$s=\frac{1}{2}$) leads to the following uniqueness theorem. 
	
	\begin{theo}\label{uniq}
		Fix~$\lambda\in\C$. Assume that~$w\neq 0$ is a generalised eigenfunction of~$\Delta$, i.\, e. an automorphic function satisfying~$\Delta w = \lambda w$ pointwise on~$\Ha$, such that~$w-w_0 \in L^2(F)$ where~$F$ is a fundamental polygon of~$X$ with cusp at~$\infty$. Then~$w$ is uniquely defined up to normalisation and addition of a cusp form~$f \in L^2_0$. 
	\end{theo} 
	
	As any eigenfunction~$v \in \Theta_\eta$ of~$\Delta^\eta_\gamma$ with~$\eta>p$ and~$\gamma\in\hat{\C}$ can be continued analytically to~$X$ in a unique way, we have the following corollary.
	
	\begin{coro}\label{uni-thm}
		Let~$\eta>p$ be fixed. For every~$\gamma\in\hat{\C}$ and~$\lambda\in\sigma_2(\Delta_\gamma)$, the eigenspace~$E_{\mathrm{(II),\gamma}}(\lambda)$ of~$\Delta_\gamma$ in~$\Theta_\eta$ associated with~$\lambda$ is one-dimensional. 
	\end{coro}
	
	Given that the Eisenstein series~$E(\, \cdot \, ,s)$ is meromorphic on the whole $s$-plane, \Cref{uni-thm} allows us to immediately write down the spectrum of $\Delta_\gamma$ in a manner similar to \cite{cdv2}, Theorem~5.
	
	Instead, we now assume that we know the Eisenstein series only on its domain of absolute convergence, and re-prove the meromorphic continuability 
	by means of holomorphic perturbation theory.

	\section{Meromorphic Continuation and Spectrum}\label{mercon}
	
	Our next step is to show that there is a one-to-one-correspondence of type-(II) eigenvalues of the set~$\{\Delta^\eta_\gamma \, | \,  \gamma\in\hat{\C} \}$ (for a fixed~$\eta>p$) on one side, and parameters~$s\in\C$ on the other.
	One direction is another immediate consequence of \Cref{uniq}.
	
	\begin{coro}\label{lameind}
		Let~$\eta>p$ be fixed. Then the following hold:
		\begin{enumerate}[(i)]
			\item For every~$\lambda\in\C$, there is at most one~$\gamma\in\hat{\C}$ such that~$\lambda \in \sigma_2(\Delta^\eta_\gamma)$.
			\item An eigenvalue $\lambda\in \sigma_2(\Delta^\eta_\gamma)$ is real if and only if~$\gamma \in \hat{\R}$ if and only if~$s$ and~$\hat{s}$ are both real or lie on the critical line. 
		\end{enumerate} 
	\end{coro}
	
	It remains to show that for any~$\lambda \in \C$, we find~$\gamma\in\hat{\C}$ such that~$\lambda \in \sigma_2(\Delta^\eta_\gamma)$, which follows from a classical result of complex analysis.
	
	\begin{prop}\label{exis-thm}
		Let~$\eta>p$ be fixed and set~$\mathcal{C}_2 = \{ (\gamma,\lambda) \ | \ \lambda \in \sigma_2(\Delta^\eta_\gamma), \gamma \in \hat{\C} \}$. Then, the projection~$\pi_2: \mathcal{C}_2 \rightarrow \C$, $(\gamma,\lambda) \mapsto \lambda$, to the second component is surjective.
	\end{prop}
	
	\begin{proof}
		As explained above,~$\mathcal{C}_2$ is a complex variety as~$\sigma_2(\Delta^\eta_\gamma)$ is the discrete spectrum of the restriction~$\Delta_\gamma|_{\Theta_\eta}$.
		Now Remmert's proper mapping theorem (\cite{Nara1966}, Theorem~VII-3) states that~$\pi_2(\mathcal{C}_2)$ is an analytic set in~$\C$ (\cite{Nara1966}, Definition~I-1) if~$\pi_2$ is closed, which follows easily from the compactness of~$\hat{\C}$. 
		
		As the only analytic sets in~$\C$ are the empty set and~$\C$ itself, the proof is completed because perturbation of any operator~$\Delta_{\gamma_0}$ provides that~$\pi_2(\mathcal{C}_2)$ contains a non-empty open subset of~$\C$.
	\end{proof}
	
	From now on, we will always assume that we have chosen local analytic functions~$\lambda(\gamma)$ of some type-(II) eigenvalue and the corresponding eigenvector~$v[\gamma]$ of~$\Delta_\gamma$. 
	
	If these functions are holomorphic in~$\gamma_0\in\C$ and~$\lambda$ can be inverted holomorphically, the composition
	\begin{equation}\label{map}
		\Psi: s \mapsto \lambda \mapsto \gamma \mapsto v,
	\end{equation}
	is holomorphic in some neighbourhood~$U$ of~$s_0$ where~$s_0$ is defined by~$\lambda_0=\lambda(s_0)$ with~$\lambda_0 \in \sigma_2(\Delta^\eta_{\gamma_0})$.

	To compute the derivative~$\lambda'(\gamma)$, we make use of the fact that the form~$\form_\gamma$ underlying~$\Delta_\gamma$ is linear with respect to~$\gamma$ and therefore can be written as Taylor series
	\begin{equation*}
		\form_\gamma(u,v) = \form^{(0)} (u,v) + \gamma \cdot \form^{(1)} (u,v)
	\end{equation*}  
	converging everywhere independently of~$u,v \in A_1$. 
	
	\begin{theo}\label{lambda_abl}
		Assume that~$\lambda(\gamma)$ and~$v[\gamma]$ are an analytic type-(II) eigenvalue and eigenvector of~$\Delta_\gamma$, respectively, in a neighbourhood of some~$\gamma_0 \in \C$. Then~$\gamma_0$ is a ramification point of~$\lambda$ if and only if~$(v[\gamma_0], \overline{v[\gamma_0]})_{L^2_\eta}= 0$. Otherwise, the derivative of~$\lambda$ in~$\gamma_0$ is given by
		\begin{equation*}
			\lambda'(\gamma_0) = \frac{v_0[\gamma_0](\eta)^2}{(v[\gamma_0], \overline{v[\gamma_0]})_{L^2_\eta}}.
		\end{equation*}
		In particular, we have~$\lambda'(\gamma)\neq 0$ for all~$\gamma\in\C$.
	\end{theo}
	
	\begin{proof}
		The eigenvalue equation implies 
		\begin{equation*}
			\form_\gamma (v[\gamma], \overline{v[\gamma]}) = \lambda(\gamma) \cdot (v[\gamma], \overline{v[\gamma]})_{L^2_\eta}. 
		\end{equation*}
		If~$\gamma_0$ is no ramification point, both the left and right hand side are holomorphic in~$\gamma_0$. The statement follows from comparing the Taylor coefficients on both sides because \Cref{lameind} provides that~$v_0[\gamma](\eta)=0$ if and only if~$\gamma=\infty$.
	\end{proof}

	\begin{rema}\label{bem-4.2}
		Due to the property of~$\form_\gamma$ having a Taylor series converging for all~$u,v \in A_1$, the family~$(\Delta_\gamma)_{\gamma\in\C}$ belongs to the special type~(B) (\cite{Kato1976}, VII-\S 4.2) as all~$\Delta_\gamma$ have compact resolvent. 
		This provides that the eigenvalues and eigenfunctions of the real-valued family~$(\Delta_\gamma)_{\gamma\in\R}$ can be perturbed in the sense that there are sequences~$(\lambda_n(\gamma))_{n\in\N}$ of eigenvalues and~$(u_n[\gamma])_{n\in\N}$ of eigenvectors of~$\Delta_\gamma$, each of which is holomorphic in a complex neighbourhood~$U_n$ of~$\R$, and such that~$(u_n[\gamma])_{n\in\N}$ is an orthonormal basis of~$L^2_\eta$ for each~$\gamma\in\R$ (\cite{Kato1976}, Remark~VII-4.2).
		
		To give a heuristic argument why the family~$(\Delta_\gamma)_{\gamma\in\hat{\C}}$ cannot be of type~(B) in any neighbourhood of~$\infty$, consider the term
		\begin{equation*}
			\form_{\nicefrac{1}{\gamma}}(u,v) = (u,v)_A + \frac{1}{\gamma} \cdot u_0(\eta) \bar{v}_0(\eta)
		\end{equation*} 
		which diverges for~$\gamma\rightarrow 0$ whenever~$u,v \in A_1 \backslash H_1(X)$. This implies that there is no way to construct~$\Delta_\infty$ based on some form~$\form_\infty$ that is defined on the space~$A_1$.
	\end{rema}
	
	\Cref{lambda_abl} implies that we can apply the inverse function theorem on~$\lambda(\gamma)$ whenever~$\gamma_0$ is neither a ramification point nor the point at infinity. 
	The following lemma provides that it is possible to avoid meeting such points by making a locally suitable choice of~$\eta$ for each~$s_0 \in \C$, so that the function~$\Psi$ in~\eqref{map} is holomorphic in a complex neighbourhood~$U$ of~$s_0$. 
	
	\begin{lemm}\label{choose_eta}
		For every~$\lambda_0\in\C$, there are~$\eta_0 >p$ and~$\gamma_0 \in \C$ such that the type-(II) eigenvalue~$\lambda(\gamma)$ of~$\Delta^{\eta_0}_{\gamma}$ is holomorphic in~$\gamma_0$ with~$\lambda(\gamma_0)=\lambda_0$.
	\end{lemm}
	
	\begin{proof}
		First consider~$s\in\C\backslash\{ \frac{1}{2} \}$ fixed with~$s(1-s) =\lambda_0$.
		Assume that~$\eta>p$,~$\gamma\in\hat{\C}$ and~$v\in L^2_\eta$ are chosen such that~$v \in E_{\mathrm{(II)}, \lambda_0}(\Delta^\eta_\gamma)$
		with constant term~$v_0(y) = a y^s + b y^{1-s}$.
		The Robin equation~\eqref{robin} allows to define~$\gamma$ as a function of~$\eta$ by the quotient
		\begin{equation*}
			\gamma_s(\eta) := -\frac{v_0'(\eta)}{v_0(\eta)}.
		\end{equation*}
		As the zero set~$Z$ of~$v_0$ is discrete, it remains to show that there is~$\eta_0 \notin Z$ such that~$\gamma_0:=\gamma_s(\eta_0)$ is no ramification point~$\lambda(\gamma)$.

		If~$s\neq 1$, this follows from the identity theorem as
		\begin{equation*}
			\gamma'_s(\eta)  = \frac{ a^2 \cdot s(2s - 1) \eta^{2s-2} + b^2 \cdot (1-s)(2s-1) \eta^{-2s} }{ v_0^2[s](\eta) }
		\end{equation*}
		cannot vanish on any neighbourhood of some fixed~$\eta_0>p$.

		If~$s=1$ or~$s=\frac{1}{2}$, \Cref{lameind} implies that~$\gamma_0 \in \hat{\R}$ is no ramification point. As the constant term~$v_0(y) = a \sqrt{y} + b \ln(y) \sqrt{y}$ for~$s=\frac{1}{2}$ has at most one zero~$\eta>0$, the lemma is proven.
	\end{proof}

	\begin{lemm}\label{par-holo}
		If the eigenfunction~$v$ of an type-(II) eigenvalue $\lambda \neq \frac{1}{4}$ depends holomorphically on a parameter~$z$, then the same holds for the coefficients~$a$ and $b$ of the constant term and all coefficients~$a_m$, $m\in\Z\backslash\{0\}$, of the Fourier series
		\begin{equation*}
			v(x+iy) = a y^s + b y^{1-s} + \sum_{m\in\Z\backslash\{0 \}}{ a_m \sqrt{y} K_{s-\nicefrac{1}{2}}(2\pi|m| y) e^{2\pi i m x} }.
		\end{equation*}
	\end{lemm}
	
	\begin{proof}
		By definition (\cite{Kato1976}, VII-\S 1.1), a vector~$v$ is holomorphic in~$z$ if and only if~$z\mapsto (v,f)$ is holomorphic for every~$f\in L^2_\eta$. 
		For~$s\notin \{ 0,1 \}$, $y_1,y_2 \in \, ]p,\eta[$ and~$f=\chi_{[y_1,y_2]}(y)$, we have 
		\begin{equation*}
			(v,f) = a \cdot \frac{y_2^{s-1} - y_1^{s-1}}{s-1} 
			- b \cdot \frac{y_2^{-s} - y_1^{-s}}{s}.
		\end{equation*}
		The same calculation for~$(v,g)$, where~$g=\chi_{[y_3,y_4]}(y)$ with~$y_3,y_4\in \, ]p,\eta[\backslash \{ y_1,y_2 \}$, gives a linear system of equations for~$a$ and~$b$ which has a unique solution for a suitable choice of the~$y_j$.
		
		The proof for~$s\in \{ 0,1 \}$ and~$m\neq 0$ is analogous after suitable choices of~$f$ and~$g$.
	\end{proof}
	
	We are now able to construct a global meromorphic function~$M[s]$ on~$\C\backslash \{ \frac{1}{2} \}$ using the locally defined functions~$\Psi$. 
	
	\begin{theo}
		There exists a unique meromorphic function~$\beta: \C \backslash \{ \frac{1}{2} \} \rightarrow \C$, $s\mapsto \beta(s)$, and a unique family $(M[s](\, \cdot \, ))_{s\in\C\backslash\{ \frac{1}{2} \}}$ of automorphic functions on~$X$ with
		\begin{equation*}
			\Delta M[s](z) = s(1-s) M[s](z) \text { for all } z\in\Ha
		\end{equation*}
		and constant term 
		\begin{equation*}
			M_0[s](y) = y^s + \beta(s) y^{1-s},
		\end{equation*}
		that is meromorphic with respect to~$s$ on~$\C\backslash \{ \frac{1}{2} \}$. The poles of~$M$ are the poles of~$\beta$. 
		These functions are symmetric with respect to complex conjugation:
		\begin{equation*}
			\overline{\beta(s)} = \beta(\bar{s}), \qquad 
			\overline{M(s)}(z) = M(\bar{s})(z) \text{ for all } z \in \Ha,
		\end{equation*}
		and fulfil the functional equations
		\begin{equation*}
			\beta(s) \beta(1-s) =1, \qquad 
			M[1-s](z) = \beta(1-s) M[s](z) \text{ for all } z \in \Ha.
		\end{equation*}
	\end{theo}
	
	\begin{proof}
		Fix~$s_0 \neq \frac{1}{2}$. Let~$\eta>p$ and~$U$ be a neighbourhood of~$s_0$ such that the map~$s\mapsto \Psi[s]$ in~\eqref{map} is holomorphic. 
		Every~$\Psi[s] \in \Theta_\eta$ is an eigenvector of some~$\Delta_\gamma$ and has a unique analytic continuation~$\tilde{\Psi}[s]$ to~$X$. 
		\Cref{par-holo} implies that the normalised function~$s\mapsto M_U[s]:= a[s]^{-1} \tilde{\Psi}[s]$ is meromorphic on~$U$ with a pole if and only if~$a[s]=0$.
		
		The symmetries of~$M_U[s]$ and~$\beta_U(s):=  a[s]^{-1}b[s]$ are implied by the symmetry of any type-(II) eigenfunction with respect to~$s\mapsto\hat{s}$ and the symmetry of~$\form_\gamma$ with respect to~$\gamma$. 
		
		\Cref{uniq} provides that the family of locally defined functions~$(M_U)_{U}$ define a global meromorphic function~$M$ on~$\C\backslash\{ \frac{1}{2} \}$. 
	\end{proof}
	
	\Cref{uniq} also implies that~$M[s]$ must be the unique continuation of the Eisenstein series~$E(\, \cdot \, , s)$ with~$\beta(s)= \ph(s)$. 
	The continuation to the whole~$s$-plane will be completed by the following proposition.
	
	\begin{prop}
		The functions~$\beta(s)$ and~$s\mapsto M[s](\, \cdot \, )$ can be continued holomorphically to~$\frac{1}{2}$ with~$\beta(\frac{1}{2}) \in \{ \pm 1 \}$ and~$M[\frac{1}{2}]=0$ if~$\beta(\frac{1}{2})=-1$.
	\end{prop}
	
	\begin{proof}
		Let~$U$ be a neighbourhood of~$s_0=\frac{1}{2}$ such that~$\beta$ has no poles in~$U\backslash\{ s_0 \}$. For all~$s\in U \backslash\{s_0\}$, the Robin parameter~$\gamma$ is given by the quotient
		\begin{equation*}
			\gamma(s) = - \frac{s \eta^{s-1} + \beta(s) (1-s) \eta^{-s}}{ \eta^s + \beta(s) \eta^{1-s} }
		\end{equation*}
		which defines a meromorphic function on~$U\backslash\{ s_0 \}$ for fixed~$\eta>p$.
		Then~$\beta$ is given by
		\begin{equation}\label{beta}
			\beta(s) = - \frac{\gamma(s) \eta^s + s \eta^{s-1}}{\gamma(s) \eta^{1-s} + (1-s) \eta^{-s}}, \qquad s\in U \backslash \{ s_0 \}.
		\end{equation}
		
		We may assume that~$U$ and~$\eta$ are chosen such that~$V:=\gamma(U) \subset \C$ and that the corresponding eigenvalue function~$\lambda$ is holomorphic on~$V$.
		
		\Cref{lambda_abl} implies that~$\lambda(\gamma)$ is invertible with inverse~$\gamma(\lambda)$ holomorphic on~$W:=\lambda(V)$ and~$\gamma(s) = \gamma(\lambda(s))$ is holomorphic on~$U$.
		So, the right hand side of~\eqref{beta} is meromorphic in~$s_0$. 
		The functional equation of~$\beta$ shows that the singularity in~$s_0$ is removable and~$\beta(s_0)=\pm 1$.
	\end{proof}

	The construction process of~$M[s]$ implies that every type-(II) eigenfunctions of the Robin Pseudo Laplacian~$\Delta_\gamma$ is given by one of the following functions if it fulfils the Robin condition~\eqref{robin}:
	\begin{enumerate}
		\item[(i)] the truncated Eisenstein series~$M^\eta[s]$ if~$s$ is no pole of~$\beta$ or if~$s=\frac{1}{2}$ and~$\beta(\frac{1}{2})=1$, 
		\item[(ii)] the conjugated truncated Eisenstein series $M^\eta[\hat{s}]$ if~$s$ is a pole of~$\beta$ (or the residue of~$M^\eta[s]$ if this pole is simple), 
		\item[(iii)] the derivative~$\partial_s M^\eta[s]|_{s=\nicefrac{1}{2}}$ of the truncated Eisenstein series if~$s=\frac{1}{2}$ and~$\beta(\frac{1}{2})=1$.
	\end{enumerate}
	
	Also, we have an analytic eigenvector map given by~$\gamma \mapsto M[s(\gamma)]$ in a neighbourhood of~$\gamma_0 =\gamma(s_0)$ whenever~$s_0\neq \frac{1}{2}$ and~$s_0$ is no pole of~$\beta$. 
	It can be shown that the formula
	\begin{equation}
		\lambda'(\gamma) = \frac{ M_0[s(\gamma)](\eta)^2 }{ ( M^\eta[s(\gamma)], M^\eta[\overline{s(\gamma)}] )_{L^2_\eta} }
	\end{equation}
	from \Cref{lambda_abl} continues holomorphically also for~$s(\gamma)=\frac{1}{2}$ or~$s(\gamma)$ being a pole, including the limit cases~$\lim_{\gamma\rightarrow\infty}{\lambda'(\gamma)}=0$ and~$\lim_{\gamma\rightarrow \gamma_0}{\lambda'(\gamma)} =\infty$ if~$\gamma_0$ is a ramification point.

	\section{Maass-Selberg Relation and Branch Points}
	
	For fixed~$\eta>p$, consider the function
	\begin{equation*}
		\gamma(s) = - \frac{s \eta^{s-1} + \beta(s) (1-s) \eta^{-s}}{ \eta^s + \beta(s) \eta^{1-s} }.
	\end{equation*}
	
	It is easily seen that $\gamma$ is meromorphic on~$\C$, all poles being simple and given by those~$s$ for which~$\lambda(s) \in \sigma_2(\Delta_\infty)$. In particular, $\gamma$ shares no pole with~$\beta$ as all the poles of~$\beta$ are removable singularities of~$\gamma$.
	For simplicity of notation, write~$P(s):= M'_0[s](\eta)$ and~$Q(s) := M_0[s](\eta)$, so
	\begin{equation*}
		\gamma(s) = -\frac{P(s)}{Q(s)}, \qquad 
		\gamma'(s) = - \frac{P'(s) Q(s) -  Q'(s) P(s)}{Q^2(s)}. 
	\end{equation*}
	
	From \Cref{lambda_abl}, we derived that
	\begin{equation*}
		\lambda'(\gamma) = \frac{Q^2(s(\gamma))}{(M[s(\gamma)], M[\overline{s(\gamma)}])_{L^2_\eta}}
	\end{equation*}
	for all~$\gamma\in\hat{\C}$. On the other hand, we have
	\begin{equation*}
		\lambda'(\gamma) = \lambda'(s) \cdot s'(\gamma) = \lambda'(s) \cdot (\gamma'(s))^{-1}.
	\end{equation*}
	
	This leads to the following formula which is a special case of the Maass-Selberg relation.
	
	\begin{prop}\label{msr}
		For all $s \in \C$ and $\eta>p$, we have 
		\begin{equation*}
			(M^\eta[s], M^\eta[\bar{s}])_{L^2_\eta} = \frac{\partial_s M'_0[s](\eta) \cdot M_0[s](\eta) - M'_0[s](\eta) \cdot \partial_s M_0[s](\eta)}{2s-1}. 
		\end{equation*}
	\end{prop}
	
	In particular, \Cref{msr} implies that the set of type-(II) eigenvalues defined by the Eisenstein series forms an orthonormal basis~$(u_n[\gamma])_{n\in\N}$ in the sense of \Cref{bem-4.2} for the restriction of~$\Delta_\gamma$ to~$\Theta_\eta$.

	The observations above enable us also to describe the occurrence of branch points. 
	
	\begin{prop}
		Let~$\eta>p$ be fixed. Then, the following statements are equivalent:
		\begin{enumerate}[(i)]
			\item $\gamma_0$ is a ramification point of order $n\geq 2$ of the local eigenvalue function $\lambda$ with $\lambda(\gamma_0)=\lambda_0 = \lambda(s_0)$, 
			\item $\gamma^{(k)}(s_0)=0$ for $k=1,\ldots, n-1$ and $s_0 \neq \frac{1}{2}$,
			\item $\lambda_0 \in \sigma_2(\Delta_{\gamma_0})$ has algebraic multiplicity~$n$, and the generalised eigenspace of rank $2 \leq k \leq n$ is spanned by the derivatives $(\partial_s)^j M^\eta[s]$, $j=1,\ldots, k$,  of the truncated Eisenstein series,
			\item the truncated Eisenstein series~$M^\eta[s_0]$ and~$M^\eta[\bar{s}_0]$ are orthogonal in $L^2_\eta$.
		\end{enumerate}
	\end{prop}
	
	\begin{proof}
		The second part of (iii) follows from repeating the argument
		\begin{equation*}
			\Delta (\partial_s M[s]) = \partial_s (\Delta M[s]) = \partial_s (\lambda(s) M[s]) = \lambda'(s) \cdot M[s] + \lambda(s) \cdot \partial_s M[s] 
		\end{equation*}
		for the higher derivatives of $s\mapsto M[s]$.
	\end{proof}

\end{document}